\documentclass[12pt]{amsart} 
\usepackage{amsfonts,graphics,amsmath,amsthm,amsfonts,amscd, amssymb,amsmath,latexsym}
\usepackage[all]{xy}
\usepackage[left=1.02in,top=1.0in,right=1.02in,bottom=1.0in]{geometry}

\theoremstyle{plain}
\newtheorem{theorem}{Theorem}[section]

\newtheorem{lemma}[theorem]{Lemma}

\newtheorem{proposition}[theorem]{Proposition}
\newtheorem{definition-lemma}[theorem]{Definition-Lemma}

\newtheorem{defn}[theorem]{Definition}
\newtheorem{notation}[theorem]{Notation}

\def\ideal#1.{I_{#1}}
\def\ring#1.{\mathcal {O}_{#1}}
\def\fring#1.{\hat{\mathcal {O}}_{#1}}
\def\proj#1.{\mathbb P(#1)}
\def\pr #1.{\mathbb P^{#1}}
\def\af #1.{\mathbb A^{#1}}
\def\Hz #1.{\mathbb F_{#1}}
\def\Hbz #1.{\overline{\mathbb F}_{#1}}
\def\pic#1.{\operatorname {Pic}\,(#1)}
\def\pico#1.{\operatorname{Pic}^0(#1)}
\def\picg#1.{\operatorname {Pic}^G(#1)}
\def\ner#1.{NS (#1)}
\def\rdown#1.{\llcorner#1\lrcorner}
\def\rup#1.{\ulcorner#1\urcorner}
\def\cone#1.{\operatorname {NE}(#1)}

\def\ccone#1.{\overline{\operatorname {NE}}(#1)}
\def\coef#1.{\frac{(#1-1)}{#1}}
\def\vit#1.{D_{\langle #1 \rangle}}
\def\mm#1.{\overline {M}_{0,#1}}
\def\H1#1.{H^1(#1,{\ring #1.})}

\def\adj#1.{\frac {#1-1}{#1}}
\def\spn#1.{\overline{#1}}
\def\ses#1.#2.#3.{0\to #1\to #2\to #3 \to 0}
\def\pek#1.#2.{\Cal P^{#1}(#2)}
\def\plk#1.#2.{\Cal P^{\leq #1}(#2)}
\def\ev#1.{\operatorname{ev_{#1}}}
\def\bminv#1.{(\nu_1,s_1;\nu_2,s_2;\dots ;\nu_{#1},s_{#1};\nu_{r+1})}
\def\zinv#1.{(\nu_1,s_1;\nu_2,s_2;\dots ;\nu_{#1},s_{#1};0)}
\def\iinv#1.{(\nu_1,s_1;\nu_2,s_2;\dots ;\nu_{#1},s_{#1};\infty)}
\def\map#1.#2.{#1 \longrightarrow #2}
\def\rmap#1.#2.{#1 \dasharrow #2}
\def\emb#1.#2.{#1 \hookrightarrow #2}

\def\Supp{\operatorname{Supp}}

\def\dim{\operatorname{dim}}
\def\length{\operatorname{length}}

\def\deg{\operatorname{deg}}

\def\Spec{\operatorname{Spec}}

\def\det{\operatorname{det}}

\def\ker{\operatorname{Ker}}
\def\im{\operatorname{Im}}

\def\e{\Cal E}

\def\e1{E_1}
\def\e2{E_2}

\DeclareMathOperator{\rk}{rk}
\DeclareMathOperator{\Lie}{Lie}

\newcommand{\sF}{\mathcal F}
\newcommand{\sG}{\mathcal G}
\newcommand{\sH}{\mathcal H}

\newcommand{\sK}{\mathcal K}
\newcommand{\sM}{\mathcal M}
\newcommand{\sN}{\mathcal N}
\newcommand{\sO}{\mathcal O}
\newcommand{\sT}{\mathcal T}

\newcommand{\bP}{\mathbb P}
\newcommand{\bQ}{\mathbb Q}

\DeclareMathOperator{\reg}{reg}

\begin{document}
\title{On the characterization of abelian varieties in characteristic $p>0$.}
\author{Christopher Hacon} 
\address{Department of Mathematics \\  
University of Utah\\  
Salt Lake City, UT 84112, USA}
\email{hacon@math.utah.edu}
\author{Zsolt Patakfalvi} 
\address{Mathematics Department\\  
Princeton University\\
Fine Hall, Washington Road \\  
Princeton, NJ 08544-1000, USA}
\email{pzs@math.princeton.edu}
\maketitle
\begin{abstract} We show that if $X$ is a smooth projective variety over an algebraically closed field of characteristic $p>0$ such that $\kappa (X)=0$
and the Albanese morphism is generically finite with degree not divisible by $p$, then $X$ is birational to an abelian variety. We also treat the cases when $a$ is separable (possibly with degree divisible by $p$) and $A$ is either supersingular or ordinary.\end{abstract}

In recent years there have been many interesting and surprisingly precise results  over  fields of characteristic 0 on the birational geometry of varieties of maximal Albanese dimension, i.e., of smooth projective varieties such that the Albanese map  is generically finite. These strong results are typically obtained by combining generic vanishing theorems and the Fourier Mukai functor. Similarly to Kodaira vanishing, it is known that generic vanishing does not directly generalize to characteristic $p>0$ (cf. \cite{HK15}). Never the less, using ideas from F-singularities, it is possible to recover some weak version of generic vanishing in characteristic $p>0$ which implies some interesting results in birational geometry (cf. \cite{HP13,Wang15,WZ14}).
The most basic varieties of maximal Albanese dimension correspond to subvarieties of abelian varieties $X\subset A$. In this case it is known that if $B$ is the connected component through the origin of the set  $\{ a\in A|a+X=X\}$, then $\bar X=X/B$ is a variety of general type
whose Gauss map is finite (\cite{Ueno73} in characteristic 0 and \cite{Abramovich94,Wei11} in char $p>0$; note that in positive characteristics some care is necessary to define Kodaira dimension).
In particular, it follows that if $\kappa (X)=0$, then $X=A$.
By a result of Kawamata and Viehweg, a  similar result is known to hold in characteristic $0$ for varieties of maximal Albanese dimension (cf. \cite{KV80}, \cite{Kawamata81}). 
The purpose of this paper is to generalize this result to positive characteristics. However, contrary to the characteristic zero results, we allow $X$ to be singular. We say for a normal variety $X$ that $\kappa(K_X) =0$ if $h^0(m K_X) = 1$ for every $m$ divisible enough. In general $K_X$ is only a Weil divisor, so we stress that we do not define Kodaira dimension of Weil divisors in general, we only define what it means for the Kodaira dimension to be $0$. We also note that in characteristic 0, the standard notation $\kappa (X)$ denotes the Kodaira dimension of a resolution which is different from $\kappa (K_X)$. Also, see \eqref{def:height} for the definition of a height of an inseparable map. 
\begin{theorem}\label{t1}
Let $X$ be a normal, projective  variety over an algebraically closed field $k$ of characteristic $p>0$. Assume that the Albanese morphism $a:X\to A$ is generically finite, or equivalently that $X$ is of maximal Albanese dimension. If $\kappa (K_X)=0$ and either
\begin{enumerate}
\item $p \nmid \deg a$, or
\item $a$ is separable, and $A$ is either ordinary or supersingular, or
\item $a$ is finite, inseparable of height $1$, and $X$ is smooth, 
\end{enumerate}
then $a$ is birational, and hence $X$ is birational an abelian variety.
\end{theorem}

In fact, in cases $(1)$ and $(2)$, the proof of \eqref{t1} goes through first establishing the following result, where we only assume separability of $a$ without any assumptions on its degree. 

\begin{theorem}\label{t2}
Let $X$ be a normal, projective  variety of maximal Albanese dimension over an algebraically closed field $k$ of characteristic $p>0$. If $\kappa (K_X)=0$ and  $a$ is separable, then $a$ is surjective.
\end{theorem}

\noindent {\bf Outline of the proof.} First, let us comment on the proof of \eqref{t2}. Suppose that $a$ is not surjective and let $Y=a(X)$, then $Y$ has positive Kodaira dimension. By \cite{Wei11} it follows (at least if $Y$ is normal) that $H^0(\omega _Y)>1$ and since $X\to Y$ is separable we obtain $H^0(\omega _X)>1$ which is impossible. If $Y$ is not normal, we use a slightly subtler argument along the same lines. 

Second we explain the main argument of case $(1)$ of \eqref{t1}. Case $(2)$ is an easy variation on this case, and case $(3)$ is based on a fairly direct understanding of the foliation corresponding to $a$. 
Note that if $A$ is ordinary, then case $(1)$ of \eqref{t1} is proven in \cite{HP13}.
To prove it in general by \eqref{t2}, we only have to show  that the degree of $a$ is $1$.

In order to establish that the degree of $a$ is 1, we proceed as follows. If $A'$ is a supersingular factor of $A$, then by a result of Oort, it is known that $A'$ is isogenous to a product of supersingular elliptic curves. By the Poincar\'e reducibility theorem, it then follows that $A$ is fibered by elliptic curves. Let $A\to B=A/E$ be the corresponding fibration and $r\circ q:X\to Z\to B$ the Stein factorization.
If $F$ is the generic fiber of $q$, then in \eqref{prop:elliptic_curve} by a subadditivity of the Kodaira dimension type argument we show that  $F$ is a smooth genus $1$ curve (c.f.  \cite{CZ13}, although since $Z$ is not smooth and $K_Z$ is not even $\bQ$-Cartier in general,  we use an alternative self-contained argument). 
Then, let $$V^0(K_X):=\{P\in \hat A|h^0(K_X+P)\ne 0\}.$$ Since $(K_X+P)|_F=K_F+P|_F\sim P|_F$, it follows that $V^0(K_X)$ is contained in finitely many torsion translates of $\hat B$. Let $C\subset A $ be the abelian subvariety generated by all elliptic curves in $A$, then by the above argument it follows that  $V^0(K_X)$ is contained in finitely many torsion translates of $\widehat {A/C}$. Since $A/C$ has no supersingular factors, we may apply the arguments similar to those of \cite{HP13}. 

%

\noindent {\bf Acknowledgement.} We would like to thank the Mathematisches Forschungsinstitut Oberwolfach where this research was initiated during  Workshop ID: 1512. The first author was partially supported by the NSF research grant no: DMS-1300750, FRG grant no: DMS-1265285 and a grant from the Simons foundation, he would also like to thank the University of Kyoto where part of this work was completed. The second author was partially supported by the NSF research grant no: DMS-1502236.

\section{Proof of the main result}

The base-field $k$ in the article is assumed to be algebraically closed and of characteristic $p>0$. 

\begin{proof}[Proof of Theorem \ref{t2}]



Assume that $a$ is not surjective. By the universal property of the Albanese morphism, then $Y=a(X)$ is not a proper abelian subvariety of $X$. Let $B$ be the biggest abelian subvariety $B \subseteq A$, such that $Y + B = Y$, let $C:= B/A$ with $\pi : A \to C$ being the quotient map, and let $Z:=\pi(Y)$. By our assumption $C \neq 0$. Consider the following commutative diagram.
\begin{equation*}
\xymatrix{
X \ar@/^2pc/[rr]^a \ar@{->>}[dr]_g \ar@{->>}[r]^f & Y \ar[d]^{\tau} \ar@{^(->}[r]^j & A \ar[d]^{\pi} \\
 & Z \ar@{^(->}[r]_{\iota} & C \\
}
\end{equation*}
Note the following:
\begin{enumerate}
 \item \label{itm:direct_sum} $\Omega_A \cong \Omega_{A/C} \oplus \pi^* \Omega_C$,
\item $Y=\pi^{-1}Z $, and hence $Y \to Z$ is smooth with fibers isomorphic to $B$,
\item \label{itm:relative_dualizing} $\omega_{A/C} \cong \sO_A$, because $\omega_A \cong \sO_A$ and $\omega_C \cong \sO_C$. 
\end{enumerate}

Let $U:=Z_{\reg}$ be the regular locus of $Z$ and $V\subset g^{-1}(U) \cap X_{\reg}$ be a dense open subset such that $f|_V$ is \'etale. Such an open set exists by the separability assumption on $a$.  Let $n:= \dim X$, $t:= \dim C$ and $s := \dim B$. By the above fact \eqref{itm:direct_sum}, $\pi^* \Omega_C^{n-s}  \otimes   \omega_{A/C}$ is a direct summand of $\Omega_A^n$. Hence, the map $a^* \Omega_A^n \to \omega_X$, which is extended reflexively from the natural homomorphism given by differentials on $X_{\reg}$, induces 
\begin{equation*}
\sO_X^{\oplus {t \choose n -s} } \cong  a^* \left(\pi^* \Omega_C^{n-s}  \otimes   \omega_{A/C} \right) \cong   \underbrace{ g^* \iota^*  \Omega_C^{n-s}}_{\textrm{fact \eqref{itm:relative_dualizing}}}   \xrightarrow{\xi}   \omega_X \\
.
\end{equation*}
We want to prove that at least two linearly independent  sections on the left are mapped to linearly independent sections.
Then it is enough to prove the above linear independence when restricted to $V$. So, let us examine $\xi|_{V}$.  
It factors as
\begin{equation*}
\xymatrix{
  a|_{V}^* \left(\pi^* \Omega_C^{n-s}  \otimes   \omega_{A/C} \right) \ar@/^15pc/[dd]^{\xi|_{V}} \ar[d]
\cong  f|_{V}^* \left(\tau^* \iota^* \Omega_C^{n-s}  \otimes  j^* \omega_{A/C} \right)  \ar[d]
\cong g|_V^* \iota^* \Omega_C^{n-s}
\\ 
f|_{V}^* \Omega_{\tau^{-1} U}^{n} 
\cong f|_{V}^* \left( \tau^* \Omega_U^{n-s}  \otimes   \omega_{A/C}|_{\tau^{-1} U} \right)   \ar[d] 
\cong g|_V^* \omega_{U}
\\
\Omega_{V}^n \cong \omega_{V}  
},
\end{equation*}
where the bottom vertical arrow is an isomorphism ($f|_V$ is assumed to be \'etale). 
Hence, it is enough to show using the identification $ \Omega_C^{n-s}|_U \cong \sO_{U}^{\oplus {t \choose n-s} }$ that the image of this ${t \choose n-s}$ dimensional space of sections via $ H^0(\Omega_C^{n-s}|_U)  \to  H^0(U, \omega_U)$ has dimension at least $2$. Since $U=Z_{\reg}$, $\Omega_C^{n-s}|_U  \to  \omega_U$ is surjective over $U$, and hence the Gauss map is defined at every point of $U$. Also, $\omega_U$ can be identified with the pullback of $\sO_{\bP^N}(1)$ via the composition of the Gauss map with the Pl\"ucker embedding, where $\bP^N$ is the target of Pl\"ucker embedding.  Furthermore, the image of the ${t \choose n-s}$ dimensional space of sections via  $H^0(\Omega_C^{n-s}|_U)  \to  H^0(U, \omega_U)$ can be identified with the pullback of $H^0(\bP^n, \sO_{\bP^n}(1))$. So, to see that this image has dimension at least $2$, it is enough to show that the image of the Gauss map is positive dimensional. However, this follows from \cite[Theorem 2, 
page 5]{Wei11}, since $Z$ is not degenerate (since we quotiented by $B$ at the beginning of the proof).
\end{proof}

\begin{proposition}
\label{prop:elliptic_curve}
Let $X$ be a normal, projective  variety of maximal Albanese dimension over $k$, such that $\kappa (K_X)=0$ and the Albanese morphism $a:X\to A$ is separable. Let $E\subset A$ be an elliptic curve, let $\pi : A \to B:=A/E$ be the quotient map and let $q:X\to Z$ and $r:Z\to B$ be the Stein factorization of $X \to B$. Then  $q$ is an isotrivial fibration such that the general fiber $F$ of $q$ maps to $E$ via an \'etale map. 
 \end{proposition}

\begin{proof}
Note that $\pi$ is smooth, and hence separable. So the induced map $X\to B$ is separable and therefore so are both $q$ and $r$. {\it First we show that $F$ is an elliptic curve, which maps via an \'etale map onto $E$}.

Since $a$ is separable, there is a dense open set $U \subseteq A$, such that $a|_{a^{-1}U} : a^{-1} U \to U$ is \'etale.  Furthermore, since there is a big open set of $A$ over which $a$ has finite fibers, $X_b \to A_b  \cong E$ is finite for a general closed point $b \in B$. Therefore, for such $b \in B$, $X_b \to A_b \cong E$ is a finite morphism of  Gorenstein projective schemes of pure dimension 1 over $k$, which is smooth over the dense open set $U_b \subseteq A_b  \cong E$. In particular, $X_b$ is reduced and every component of $X_b$ is separable over $A_b  \cong E$. Since general closed fibers of $q$ are  irreducible (e.g., \cite[Lemma 9]{Kol03}) closed $1$-dimensional subschemes of general closed fibers of $X \to B$, general fibers of $q$ are integral and separable over $E$. Let $F$ be such a general fiber and $\widetilde{F}$ its normalization. Since $\widetilde{F} \to E$ is a separable morphism of smooth curves,  $\widetilde{F} \not\cong \mathbb{P}^1$. Furthermore, if $\widetilde{F}$ is an 
elliptic curve, then $\widetilde{F} \to E$ has to be \'etale by the Hurwitz formula, and then $\widetilde{F} \to F$ is the identity. In particular, in this case we proved our goal. Therefore, we may assume that the genus $g$ of $\widetilde{F}$ is at least $2$. 

Our goal is to prove a contradiction by showing that $h^0(m K_X) \geq 2$ for some integer $m>0$. Let $q' : X' \to Z'$ be a flattenification  of $q$, that is, $Z' \to Z$ is a birational morphism of normal varieties, $X'$ is a normal variety mapping birationally onto the main component of $X \times_Z Z'$, and $q'$ is equidimensional \cite[Thm 5.2.2]{RG71}. Let $Z_0 \subseteq Z'$ be the regular locus, $X_0:=q'^{-1} Z_0$ and let $q_0: X_0 \to Z_0$ be the induced morphism. Note that $Z_0$ is a big open set of $Z'$ and hence by the equidimensionality of $q'$, so is $X_0$ in $X'$. In particular, for each integer $m>0$, 
\begin{equation*}
h^0(X,m K_X) 
\geq \underbrace{h^0(X',m K_{X'})}_{\parbox{138pt}{ \tiny $\xi : X' \to X$ is birational, so \qquad \qquad $\xi_* : H^0(X',mK_{X'}) \hookrightarrow H^0(X, mK_X)$}} 
= \underbrace{h^0(X_0,m K_{X_0})}_{\textrm{ $X_0 \subseteq X'$ is a big open set}}. 
\end{equation*}
Hence, to obtain a contradiction with the assumptions $g>1$ and $\kappa(K_X)=0$, it is enough to show that $h^0(X_0, m K_{X_0}) >1$ for some integer $m>0$. Furthermore, note that since $X$ is separable over $B$, so is $Z_0$, and hence $h^0(Z_0, m K_{Z_0}) >0$ for every positive $m>0$. \emph{Hence, in fact it is enough to show that $h^0(X_0, m K_{X_0/Z_0}) >1$ for some integer $m>1$.} The rest of the proof, except the last paragraph about isotriviality, contains a proof of the latter statement. Note that since $Z_0$ is regular, $K_{Z_0}$ is a Cartier divisor, and hence $K_{X_0} = K_{X_0/Z_0} + q_0^* K_{Z_0}$ holds. This is  the main reasons why the step of passing to the flattenification is indispensible for the current line of proof. 

\emph{We claim  that after possibly further shrinking $Z_0$ (but still keeping it big in $Z'$, which is the only fact used about $Z_0$ above and is always assumed in the rest of the proof), there is a diagram 
\begin{equation*}
\xymatrix{
X_0 \ar[d]^{q_0} & \ar[l]^{\delta} \ar[d] X_0^n  &    \widetilde{X}_0 \ar[l]^\rho \ar[r]^{\tau}  \ar[d]^{\widetilde{q}_0} & \overline{X}_0  \ar[d]^{\overline{q}_0} \\
Z_0 & \ar[l] \widetilde{Z}_0  & \widetilde{Z}_0 \ar@{=}[l] \ar@{=}[r] & \widetilde{Z}_0 
}, 
\end{equation*}
where 
\begin{itemize}
\item $q_0$ is a flat map from a normal variety onto a smooth variety  with irreducible general fibers,
\item $ \widetilde{Z}_0$ is smooth, $\widetilde{Z}_0 \to Z_0$ is a finite morphism, and $X^n_0$ is the normalization of $X_0 \times_{Z_0} \widetilde{Z}_0$,
\item $\widetilde X _0$ is smooth, $\rho$ is birational, and $\widetilde{q}_0$ is a flat family of curves with smooth general fibers,
 \item $\overline{X}_0$  has canonical singularities and $\overline{q}_0$ is a family of stable curves of genus $g $.
\end{itemize}}
The above claim can be derived  from \cite{dJ97} as follows: by \cite[Thm 2.4 with condition (vii)(b)]{dJ97}, there is an alteration $q_0^1 : X_0^1 \to Z_0^1$ of $q_0$, which is projective, semi-stable \cite[2.16]{dJ96}, and the general fiber is a smooth genus $g$ curve.  Alteration here means a diagram as follows with the horizontal arrows being alterations.
\begin{equation*}
\xymatrix{
X_0 \ar[d]^{q_0} & X_0^1 \ar[d]^{q^1_0} \ar[l] \\
Z_0   & Z_0^1 \ar[l]
}
\end{equation*}
Next, by \cite[Lem 3.7]{dJ97} there is an alteration $Z_0^2 \to Z_0^1$ and a family $X_0^2 \to Z_0^2$ of stable curves which is isomorphic over a dense open set of $Z_0^2$ to $X_0^1 \times_{Z_0^1} Z_0^2 \to Z_0^2$. Let $W$ be the closure of the graph of this birational isomorphism in $X_0^1 \times_{Z_0^1} X_0^2\cong  (X_0^1 \times_{Z_0^1}{Z_0^2)}\times_{Z_0^2} X_0^2$. Applying \cite[Thm 2.4 with condition (vii)(b)]{dJ97} to $W \to Z_0^2$ yields an alteration $\widetilde{q}_0 : \widetilde{X}_0 \to \widetilde{Z}_0$ of $W \to Z_0^2$, which is a semi-stable family of curves. Furthermore, by construction $\widetilde{X}_0$ maps to $X_0^2 \times_{Z_0^2} \widetilde{Z}_0$ birationally over $\widetilde{Z}_0$. Hence, we may define  $\overline{X}_0 :=X_0^2 \times_{Z_0^2} \widetilde{Z}_0$. The factorization of $\widetilde{X}_0 \to X_0$ through $X_0^n$ is due to the fact that $\widetilde X _0$ is smooth and hence normal. By possibly further restricting $Z_0$ (but 
keeping it big in $Z'$) we may assume that $\widetilde{Z}_0 \to Z_0$ is finite. We are left to show that $\overline{X}_0$ has canonical singularities. So, let us fix a singular closed point $x \in \overline{X}_0$, and let $z$ be its image in $\widetilde{Z}_0$. Then there is necessarily a node at $x$ in the fiber over $z$. In particular, $\Spec \widehat{\sO_{\overline{X}_0,x}} \to \Spec \widehat{\sO_{\widetilde{Z}_0,z}}$ can be pulled back from the universal deformation space of the node: $$\Spec k [[ x,y,t ]] /(xy-t) \to \Spec k [[ t ]].$$ 
In particular, $\widehat{\sO_{\overline{X}_0,x}} \cong k [[ x,y,t_1, \dots, t_n ]]/(xy- g(t_1, \dots, t_n))$ 
for some $g(t) \in k [[ t_1, \dots, t_n ]]$. 
Note that this is a Gorenstein singularity and using Fedder's criterion, it is easy to see that $X$ is strongly $F$-regular at $x$: one just has to verify that there is an integer $e>0$, such that $(xy -g(t))^{p^e-1}$ contains a non-zero monomial, in which the coefficients of each variable $x, y, t_1, \dots, t_n$ is smaller than $p^e -1$, in fact if $at^b$ is a non-zero monomial of $g(t)$, then $-ax^{p^e-2}y^{p^e-2} t^b$ will be a non-zero monomial of $(xy -g(t))^{p^e-1}$, all the coefficients of which are smaller than $p^e-1$ whenever $p^e-1>b$. In particular, $\overline{X}_0$ is Gorenstein klt at $x$, which implies canonical. This finishes the proof of our claim.

Using our claim, let $U^n:= \delta^{-1} \left((X_0)_{\mathrm{reg}} \right)$.  According to \cite[Lemma 8.8]{KP15}, $\left.\delta^* K_{X_0/Z_0}\right|_{U^n} \geq \left.K_{X^n_0/\widetilde{Z}_0}\right|_{U^n}$ (for any compatible choice of relative canonical divisors). However, since $\delta$ is a finite map, $U^n$ is a big open set of $X^n_0$, and hence the restriction to $U^n$ can be dropped: $\delta^* K_{X_0/Z_0} \geq K_{X^n_0/\widetilde{Z}_0} $. \emph{We claim that it is enough to show that $h^0\left(X^n_0, m K_{X^n_0/\widetilde{Z}_0} \right) >1 $ for some integer $m>1$}.  Indeed, if the latter holds, then $\kappa \left( K_{X^n_0/\widetilde{Z}_0} \right) >0$, and then $\kappa(\delta^* K_{X_0/Z_0}) >0$, and eventually $\kappa( K_{X_0/Z_0})>0$ by the invariance of Kodaira dimension under pullbacks via finite maps. 

Since $\rho$ is birational, the above claim yields that it is enough to show that $h^0\left(\widetilde{X}_0, m K_{\widetilde{X}_0/\widetilde{Z}_0}\right) >1 $ for some integer $m>1$. Since $\overline{X}_0$ has canonical singularities and $\tau$ is birational, 
 $h^0\left(\widetilde{X}_0, m K_{\widetilde{X}_0/\widetilde{Z}_0}\right) = h^0\left(\overline{X}_0, m K_{\overline{X}_0/\widetilde{Z}_0}\right)$ for every integer $m>0$. By  \cite[Thm 0.4]{Kee99} and \cite[2.12]{CZ13},  $K_{\overline{X}_0/\widetilde{Z}_0}$ is semi-ample. Since $g>1$, for the general fiber $\overline{F}$ of $\overline{q}_0$,   $\deg (K_{\overline{X}_0/\widetilde{Z}_0}|_{\overline{F}})=2g-2>0$ and so   $\kappa\left(K_{\overline{X}_0/\widetilde{Z}_0}\right) \geq 1$. Hence, we derived a  contradiction from our original assumption $g>1$, which concludes the proof of step $2.a$.

{\it Lastly, we show that the fibration $q$ of  elliptic curves is isotrivial.}
 Assume $q$ is not isotrivial. Then it stays non-isotrivial after passing to an uncountable algebraically closed base-field. The reason, is that $q$ being non isotrivial is equivalent to the corresponding moduli map having positive dimensional image, which is a property invariant under field extension. However, then we have uncountably many non-isomorphic fibers, which are all isogenous to a fixed elliptic curve. This is a contradiction 
\end{proof}

Let us recall the notion of cohomology support loci,  a standard tool of generic vanishing theory. 

\begin{defn}
For a projective variety $X$ with Albanese map $a : X \to A$, we define
\begin{equation*}
V^0(K_X) := \{P \in \hat A | h^0(K_X + a^* P)>0 \} .
\end{equation*}
Similarly, if $\sM$ is a coherent sheaf on an abelian variety $A$, then 
\begin{equation*}
V^0(\sM) := \{P \in \hat A | h^0(\sM \otimes P)>0 \} .
\end{equation*}
Obviously, $V^0(K_X)=V^0(a_* \omega_X)$ using these notations.
\end{defn}

\begin{lemma}
\label{lem:V_0_K_X}
Let $X$ be a normal, projective  variety of maximal Albanese dimension, such that $\kappa (K_X)=0$ and the Albanese morphism $a:X\to A$ is separable. Then the only torsion closed point in $V^0(K_X)$ is $\sO_{\hat A}$.
\end{lemma}

\begin{proof}
By separability, there is an embedding $\omega_A \hookrightarrow a_* \omega_X$ (the differentials give $a^* \omega_A \to \omega_X$,  the pushforward of which composed with the natural embedding $\omega_A \to a_* a^* \omega_A$ is the above map). This shows that $\sO_{\hat A} \in V^0(K_X)$. Now assume that $P \in V^0(K_X)$ for some torsion point $P \neq \sO_{\hat A}$. Then there are $D \in |K_X + P|$ and $E \in |K_X|$. Choose $n$, such that $nP \sim 0$. Then $nD = nE$, since $h^0(nK_X) \leq 1$ by the $\kappa(K_X)=0$ assumption. In particular, $D=E$, which contradicts the assumption $P \neq \sO_{\hat A}$. 
\end{proof}

From here we use the theory developed in \cite{HP13}, in particular the notion of Cartier modules. Recall that a Cartier module is a triple $(\sM,\phi,e)$, where $e >0$ is an integer, $\sM$ is a coherent sheaf on a scheme of positive characteristic, and $\phi$ is an $\sO_X$-linear map $\phi : F^e_* \sM \to \sM$. One can iterate the structure map to obtain $\phi^{e'} : F^{e \cdot e'}_* \sM \to \sM$. It is known that $\im \phi^{e'}$ is the same for all $e' \gg 0$, and this image is called the stable submodule $S^0 \sM$. If $S^0 \sM=0$, then $\sM$ is called nilpotent. For the Cartier modules used in this article $e=1$. For convenience also we usually drop $\phi$ and $e$ from the notation of a Cartier module.

\begin{notation}
\label{notation:Cartier}
For a Cartier module $\sM$ on an abelian variety $A$, we introduce the following notation for any integer $e>0$:
\begin{itemize}
 \item $\sM_e:= F^e_* \sM$
\item $\Lambda _{\sM_e}=R\hat S D_A(\sM _e)$, where $R\hat S$ and  $D_A$ are the Fourier-Mukai and the dualizing functors of $A$, 
\item $\Lambda_{\sM}={\rm hocolim}\Lambda _{\sM_e}$.
\end{itemize}
Note that by \cite[3.1.2]{HP13} $\Lambda_{\sM}$ is in fact, the ordinary direct limit of the sheaves $\sH^0(\Lambda _{\sM_e})$, so one may disregard ${\rm hocolim}$ from the above definition. 

If $\sM = a_* \omega_X$, for the Albanese morphism $a : X \to A$ of a normal, projective variety,  then we use the short hand notation $\Lambda _{e}$ and $\Lambda$ for $\Lambda _{\sM_e}$ and $\Lambda _{\sM}$, respectively. 
\end{notation}

\begin{proposition}
\label{prop:Lambda}
Let $X$ be a normal, projective  variety of maximal Albanese dimension over $k$, such that $\kappa (K_X)=0$ and the Albanese morphism $a:X\to A$ is separable. Then, using notation \eqref{notation:Cartier}, the only point $Q$ where the map of stalks  $\sH^0(\Lambda_e)_Q \to \Lambda_Q$ is possibly not zero are the $p^e$ torsion points of $\hat A$. 
\end{proposition}

\begin{proof}

Let $C\subset A$ be the abelian subvariety generated by all elliptic curves $E\subset A$ and $\pi : A\to B=A/C$ the induced morphism. \emph{We claim that  $B$ has no supersingular factors and  $V^0(K_X)$ is contained in finitely many torsion translates of $\hat B\subset \hat A$.}

Indeed, by a result of Oort \cite[Theorem 4.2]{Oort74}, $A$ is supersingular if and only if it is isogenous to a product of supersingular elliptic curves. If $B$ has a supersingular factor then by Poincar\'e reducibility, $B$ is isogenous to a product of an elliptic curve and an abelian variety. But then $B$ contains an elliptic curve say $E\subset B$. By Poincar\'e reducibility, then $A$ contains an elliptic curve $E'$ with $\pi _*E'\ne 0$. This contradicts the definition of $C$.

Suppose now that $D\in |K_X+P|$. For any elliptic curve $E\subset A$, we consider the Stein factorization $X\to Z\to A/E$. According to \eqref{prop:elliptic_curve}, for general $z \in Z$, $X_z\to E$ is an isogeny and so $X_z $ is an elliptic curve.
It follows that $0\leq D|_{X_z}\sim (K_X+P)|_{X_z} =P|_{X_z}$ implies that $P\in {\rm Ker}(\hat A\to \hat X_z)$. Since $X_z\to E$ is 
an isogeny and the union of these elliptic curves generates $C$, it follows that $P$ belongs to finitely many torsion translates of $\hat B ={\rm Ker}(\hat A\to \hat C)$. This concludes our claim.

Now, we conclude the proof itself. Since $V^0(K_X)$ is contained in finitely many torsion translates of an abelian subvariety $\hat B \subset \hat A$ such that $\hat B$ has no supersingular factors, \cite[proof of 3.3.5]{HP13} implies the following:  for each integer $ e \geq 0$, each maximal dimensional irreducible component of the set of points $Q\in \hat A$ such that the image of $\mathcal H^0(\Lambda _e)_Q \to \Lambda _Q $ is non-zero, is a torsion translate of an abelian subvariety of $\hat B$. Since the only torsion points these loci contain are the $p^e$-torsion points according to \eqref{lem:V_0_K_X} and cohomology and base-change (cf. \cite[3.2.1]{HP13}), the above locus is contained in the set of $p^e$-torsion points.

\end{proof}

\begin{lemma}
\label{lem:Cartier_kernel}
If $N \to K$ is morphism of Cartier modules, then $M:= \ker(N \to K)$ is also a Cartier module. 
\end{lemma}

\begin{proof}
Let $e>0$ be an integer such that some powers of the structure homomorphism of $N$ and $K$ appear as $F^e_* N \to N$ and $F^e_* K \to K$. Consider then the following commutative diagram.
\begin{equation*}
\xymatrix{
0 \ar[r] & F^e_* M \ar[r] \ar@{-->}[d] & F^e_* N \ar[r] \ar[d] & F^e_* K \ar[d] \\
0 \ar[r] &  M \ar[r] & N \ar[r] &  K \\
}
\end{equation*}
We see that the two solid vertical arrows induce the dashed vertical arrow by diagram chasing: if $m \in F^e_* M$, then  its image in $F_*^e K$ is zero, and hence so is it in $K$, however then its image in $N$ lands in the submodule $M \subseteq N$. 
\end{proof}

\begin{proof}[Proof of case $(1)$ of \eqref{t1}]
By the separability assumption, there is an embedding $\omega_A \to a_* \omega_X$, and by the degree assumption the trace map $a_* \omega_X \to \omega_A$ splits this, after a multiplication of the latter by a non-zero constant. The trace map $a_* \omega_X \to \omega_A$ is compatible with the Frobenius traces, by the compatibility of the absolute Frobenii maps with $a$ and by the functoriality of traces. The map $\omega_A \to a_* \omega_X$ is also compatible with the Frobenius traces in the adequate sense according to \cite[Step 1 of the proof of Thm 4.3.1]{HP13}. Therefore, using \eqref{lem:Cartier_kernel}, the above splitting is a splitting of Cartier modules, say $a_* \omega_X \cong \omega_A \oplus \sM$. According to \eqref{lem:V_0_K_X}, every point of $V^0(\sM)$ is non-torsion.  Then, by cohomology and base-change (cf. \cite[3.2.1]{HP13}), $\sH^0(\Lambda_{\sM_e})$ is supported on non-torsion points. Combining this with \eqref{prop:Lambda}, we obtain that the map $\sH^0(\Lambda_{\sM_e}) \to \Lambda_{\sM}$ 
is the zero map. It follows then 
that $\Lambda_{\sM}=0$. This implies that $\sM$ is nilpotent according to \cite[3.1.1 and 3.2.4]{HP13}. But then $\varprojlim F^e_* a_*\omega _X\cong \varprojlim F^e_*\omega _A$. Since the image of $\varprojlim F^e_* a_*\omega _X\to a_*\omega _X$ is   the stable submodule $S^0 a_* \omega_X \subseteq a_* \omega_X$ it follows that $S^0 a_* \omega_X$ is contained in $\omega_A \subseteq a_* \omega_X$. On the other hand, $S^0 a_* \omega_X|_U = a_* \omega_X|_U$  over any open set $U \subseteq A$ over which $X$ is regular and $a$ is finite.  Hence, the rank of $a_* \omega_X$ is $1$, and consequently $\deg a = 1$.
\end{proof}

Next we prepare the proof of case $(2)$ of \eqref{t1}. We use the Verschiebung morphism  $V$ of $A$, which is the dual of the (relative) Frobenius morphism, and $V$ is \'etale if and only if $A$ is ordinary \cite[2.3.1,2.3.2]{HP13}. In particular, the Cartier module structure of a Cartier module $\sM$ translates to a ``dual Cartier module'' structure $R \hat S D_A(\sM) \to V^* R \hat S D_A(\sM)$ on $R \hat S D_A(\sM)$. The understanding of this ``dual Cartier module'' structure is the key of our approach. We note that strictly speaking $V$ is a morphism $A \to A'$, where $A'$ is the twist of $A$, that is, the $k$ structure is adequately composed with a Frobenius of $k$. Below, we will regularly consider $V$ as a morphism $A \to A$ by twisting back the $k$ structure on the target. This way $V$ is not $k$-linear anymore, but it becomes an endomorphism. 

\begin{lemma}
\label{lem:quotient}
Let $A$ be an ordinary abelian variety. Let $\sF$ be a coherent sheaf supported at $0 \in \hat A$, endowed with a map $\sF \to V^* \sF$, which is an isomorphism at $0$. Then there is a quotient sheaf $\sG \cong k(0)$ of $\sF$ with a map $\sG \to V^* \sG$, such that  $\sG \to V^* \sG$ is an isomorphism at 0, and the following diagram is commutative
\begin{equation*}
\xymatrix{
\sF \ar[r] \ar@{->>}[d] & V^* \sF \ar@{->>}[d] \\
\sG \ar[r] & V^* \sG
}
\end{equation*}

\end{lemma}

\begin{proof}
 Denote by $\xi$ the map $\sF \to V^* \sF$. Then 
\begin{equation*}
\xi \otimes k(0)  : \sF \otimes k(0) \to (V^* \sF) \otimes k(0) \cong (V \otimes k(0))^* ( \sF \otimes k(0)) \cong  \sF \otimes k(0) 
\end{equation*}
 can be identified with a $p$-linear automorphism  of the finite dimensional vector space $\sF \otimes k(0)$ over $k$, where $V \otimes k(0)$ is the morphism $\Spec k(0) \to \Spec k(0)$ induced by $V$. Here we are using that $V$ is \'etale. However, since $k$ is assumed to be algebraically closed,  $\sF \otimes k(0)$ admits a one-dimensional quotient $W$, to which $\xi \otimes k(0)$ descends as an isomorphism. Let $\iota: 0 \hookrightarrow \hat A$ be the inclusion of the origin. Then the composition
\begin{equation*}
\sF \twoheadrightarrow \iota_* (\sF \otimes k(0)) \twoheadrightarrow \iota_* W =: \sG
\end{equation*}
yields a quotient as desired. 
\end{proof}

\begin{lemma}
\label{lem:nice_Cartier_module}
 Let $A$ be an ordinary abelian variety. Let $\sF$ be a coherent sheaf supported at $0 \in  \hat A$, endowed with a map $\sF \to V^* \sF$, which is an isomorphism at $0$. Then the induced Cartier module on $[-1_A]^*D_ARS(\sF)[-g]$ is an iterated extension of the Cartier module $\omega_A$ up to allowing twistings the structure map by non-zero elements of $k$ in the $\omega_A$ factors. 
\end{lemma}

\begin{proof}
Just for the course of the proof, let us call  Cartier modules with the above structure (i.e., iterated extension of \dots) \emph{``nice''} Cartier modules. 

We show the statement by induction on $\length \sF$. If $\length \sF=1$, then $\sF \to V^* \sF$ can be identified with $R \hat S (D_A(\omega_A)) \to R \hat S (D_A(F_* \omega_A))$ up to a multiplication by a non-zero element of $k$. 

If $\length \sF>1$, then we work by induction. Let $\sG$ be the quotient of \eqref{lem:quotient}, and let $\sH$ be the kernel of $\sF \to \sG$. Then we have an induced commutative diagram with exact rows:
\begin{equation*}
\xymatrix{
0 \ar[r] & \sH \ar[d] \ar[r] & \sF \ar[d] \ar[r] & \sG \ar[d] \ar[r] & 0 \\
0 \ar[r] & V^* \sH \ar[r] & V^* \sF \ar[r] & V^* \sG \ar[r] & 0
}
\end{equation*}
Applying $[-1_A]^*D_ARS(\_)$ we obtain the following exact sequence of Cartier modules.
\begin{equation*}
\xymatrix{
0 \ar[r] &[-1_A]^*D_ARS(\sG) \ar[r] & [-1_A]^*D_ARS(\sF) \ar[r] &  [-1_A]^*D_ARS(\sH) \ar[r] & 0
} 
\end{equation*}
The Cartier modules $[-1_A]^*D_ARS(\sG)$ and $[-1_A]^*D_ARS(\sH)$ are ``nice'' by induction. Hence so is $[-1_A]^*D_ARS(\sF)$. 
\end{proof}

\begin{lemma}
\label{lem:sub_Cartier_module_not_nilpotent}
Let $X$ be a smooth variety, and let $\sM$ be a sub-Cartier module of a Cartier module $\sN$, where $\sN$ is the iterated extension of the Cartier module $\omega_X$ with itself, allowing the structure map of $\omega_X$ in these pieces to be multiplied by a non-zero element of $k$. If $\sM \neq 0$, then $\sM$ is not nilpotent. 
\end{lemma}

\begin{proof}
We show the statement by induction on $\rk \sN$. If $\rk \sN=1$, then $\omega_X \cong \sN$. By torsion-freeness of $\omega_X$, $\sM$ is generically equal to $\omega_X \cong \sN$, and then the statement follows by the smoothness of $X$. 

Hence, we may assume that $\rk \sN>1$. By the assumption, we have an exact sequence of Cartier modules as follows.
\begin{equation*}
\xymatrix{
0 \ar[r] & \omega_X \ar[r] & \sN \ar[r] & \sK \ar[r] & 0
}
\end{equation*}
If $\sM \cap \omega_X = 0$, then $\sM$ embeds into $\sK$, and then we are done by induction. On the other hand, if $\sM \cap \omega_X \neq 0$, then first by restricting to a non-empty open set, we may assume that $\sM \supseteq \omega_X$. Second, 
since $\omega_X$ is locally split, we may lift every local section of $ \omega_X \subseteq \sM$ via this local splitting to local sections of $F^e_* \omega_X \subseteq F^e_* \sM$. This concludes our proof. 
\end{proof}

\begin{proof}[Proof of case $(2)$ of \eqref{t1}]
\emph{First, we treat the case when $A$ is ordinary}. For each integer $e \geq 0$, set 
\begin{equation*}
\Lambda _e':={\rm Im}\left(\mathcal H ^0(\Lambda _e)\to \Lambda \right)\in {\rm QCoh} (\hat A ), 
\end{equation*}
 which is a coherent sheaf.
In particular $\Lambda_e'\to \Lambda_{e'}'$ is an injective homomorphism of Artinian coherent sheaves for any $0\leq e\leq e'$.  According to \eqref{prop:Lambda},  $\Lambda '_e$ is supported  at the $p^e$ torsion points of $\hat A$. Then applying $RS$ and $D_A$ to the induced homomorphism (in $D(\hat A)$) $\Lambda _e\to \Lambda '_e$, we obtain natural maps $$V_e:=[-1_A]^*D_ARS(\Lambda _e')[-g]\to F_*^e a_*\omega _X.$$
Since $V _e=F^e_*V _0$ and the induced maps $V _e\to V _{e-1}$ are compatible with the trace maps $F_*^ea_*\omega _X\to F_*^{e-1}a_*\omega _X$, it follows that the above homomorphisms $\psi _e : V_e\to F^e_*a_*\omega _X$ induce a homomorphism of Cartier modules. 

We claim that \emph{$\Lambda_0' \to V^* \Lambda_0' = \Lambda_{1}'$ is an isomorphism} at 0, which implies that the Cartier module $V_1=F_* V_0 \to V_0$ has the structure described in \eqref{lem:nice_Cartier_module}.  Indeed, by the ordinarity assumption $V$ is \'etale. Hence, $\length \left(V^* \Lambda_0'\right)_0 = \length \left(\Lambda_0'\right)_0$, where the outside $0$-subindex means taking stalk at $0 \in \hat A$. However, as noted above, we know that $\Lambda_0' \to V^* \Lambda_0'$ is injective, and therefore so is $\left(\Lambda_0'\right)_0 \to \left(V^* \Lambda_0'\right)_0$. This implies our claim using the above equality of lengths.

According to \eqref{lem:Cartier_kernel}, $\mathcal M:= \ker \psi_0$ inherits a Cartier module structure. 
Since $\mathcal H ^0(\Lambda _e)\to \Lambda '_e$  is surjective, by cohomology and base change $H^0( F_*^e a_*\omega _X\otimes P)^\vee \to H^0(V_e\otimes P)^\vee $ is surjective for all $P\in \hat A$ (cf. \cite[3.2.1]{HP13}) and hence $H^0(V_e\otimes P)\to H^0( F_*^e a_*\omega _X\otimes P)$ is injective. In particular, it follows that $V^0(\mathcal M)=\emptyset$.
By \cite[3.2.1]{HP13} it follows that $\Lambda _{\mathcal M}=0$, where $\Lambda _{\mathcal M}$ is as defined in \eqref{notation:Cartier}. In particular, $\sM$ is a nilpotent Cartier module. 
However, using the ordinarity assumption, by \eqref{lem:sub_Cartier_module_not_nilpotent}, then $\sM=0$ follows. Therefore, $\psi_e$ is an injection, and then $V_e$ can be regarded as a submodule of $F^e_* a_* \omega_X$. Using that according to \cite[3.1.4]{HP13} $\varprojlim V_e = \varprojlim F^e_* a_* \omega_X$, and that the image of the latter in $a_* \omega_X$ is $S^0 a_* \omega_X$, we obtain that $V_0 = S^0 a_* \omega_X$ (we are also using that $V_e \to V_{e-1}$ is surjective). Furthermore, 
since $S^0a_*\omega _X$ is  isomorphic to $a_*\omega _X$ in codimension $1$, it follows that $c_1(a_*\omega _X)=0$.
Then, since $a$ is separable, it follows that $X\to A$ is \'etale over a big open subset of $A$.  By the purity of the branch locus theorem \cite[Expos\'e X, Theorem 3.1]{SGA1}, the Stein factorization of $X\to A$ is \'etale. Since $a$ is the Albanese morphism, $a$ is birational.

\emph{Suppose now that $A$ is supersingular}, so that by a result of Oort, $A$ is isogenous to a product of supersingular elliptic curves (see \cite{Oort74}). In particular, for some integer $n >0$ and each $1 \leq j \leq \dim X$, we have a commutative diagram as follows, where the horizontal arrows are isogenies.
\begin{equation*}
\xymatrix{
\times_i E_i \ar[r]^f \ar@/^1.5pc/[rr]^{[n]} & A  \ar[d]^{\pi_j} \ar[r]^{\eta} \ar[dr]^{\xi_j} & \times_i E_i \ar[d] \\
 & B_j:=A/f(E_j) \ar[r] & \times_{j \neq i} E_i =:C_j
}
\end{equation*}
 Also, we may write $K_X = a^* K_A + R = R$ for some effective divisor $R$ on $X$. According to \eqref{prop:elliptic_curve}, there is an effective divisor $D_j$ on $B_j$, such that $\Supp a_* R \subseteq \Supp \pi_j^* D_j$.  However, then since the map $B_j \to C_j$ is finite, there is also an effective divisor $G_j$ in $C_j$, such that $\Supp a_* R \subseteq \Supp \xi_j^* G_j = \Supp \eta^* (E_j \times G_j)$. In particular, 
\begin{equation}
\label{eq:ramficiation}
\Supp \eta_* a_* R \subseteq \bigcap_j \Supp ( E_j \times G_j). 
\end{equation}
 Note that $\bigcap_j \Supp( E_j \times G_j)$ cannot have a divisorial component. Indeed, if it had one, then it would be of the form $E_j \times H_j$ for every $j$, where $H_j$ is a prime divisor of $C_j$. However this is impossible. In particular, $\eta_* a_* R=0$, and then using that $\eta$ is finite, $a_* R =0$. Hence, $a$ is \'etale over a big open set of $A$. From here we conclude as in the last sentence of the ordinary case. 

\end{proof}

\begin{defn}
\label{def:height}
Let $f : X \to Y$ be a finite, purely inseparable map  between integral varieties over $k$. Then the height of $f$ is the smallest integer $e>0$, such that $K(X) \subseteq K(Y)^{1/p^e}$. Note that by the purely inseparable assumption, there is such an integer.
\end{defn}

\begin{proof}[Proof of case $(3)$ of \eqref{t1}]
The assumptions are equivalent to the factorization
\begin{equation*}
\xymatrix{
A' \ar[r] \ar@/^1.0pc/[rr]^F & X \ar[r]_a & A.
}
\end{equation*}
Here by $F$ we in fact mean the relative Frobenius over $k$ and by putting prime on the space we denote the appropriate pullback via the Frobenius of $k$ that makes it the source of the relative Frobenius morphism. 

According to \cite[p 105]{Eke88}, $X=(X/\sF)'$ for some $1$-foliation $\sF \subseteq \sT_A \cong \sO_A \otimes_k \Lie(A)$ (we refer to \cite[p 96-]{Mum70} on Lie algebras, and note that $\sF$ is a saturated and hence reflexive subsheaf of  $\sT_A$).  
Further, by \cite[Prop 1.1.(iii)]{Eke88}, $\omega_X^p \cong a^* \left(\omega_A \otimes (\det \sF)^{1-p} \right)$. Hence $\kappa(\det \sF ^\vee ) =0$. However, 
\begin{equation*}
\det \sF \subseteq \wedge^{\rk \sF} \sT_A \cong \sO_A \otimes_k \wedge^{\rk \sF} \Lie(A). 
\end{equation*}
Hence $\det \sF \cong \sO_A(-D)$ for some effective divisor $D$. However, $\kappa(\sO_A(D))=0$ can hold only if $D=0$. Hence, $\det \sF \cong \sO_A$. 

\emph{We claim that if $\sF \subseteq \sO_A^{\oplus j}$ is a reflexive coherent subsheaf with $\det \sF =0$, then $\sF \cong \sO_A^{\oplus i}$ for some integer $i \geq 0$.} Indeed, if there is a factor $\sO_A$ of $\sO_A^{\oplus j}$, such that $\sF \cap \sO_A =0$, then $\sF$ embeds also into $\sO_A^{\oplus (j-1)}$. Hence, we may assume that $\sF \cap \sO_A \neq 0$ for all the $j$ factors. However, by the reflexivity of $\sF$ (i.e. all sections extend from big open sets) in this case $\sF \cap \sO_A = \sO_A(-D)$ for any fixed factor, where $D$ is an effective divisor. Let $\sO_A^{\oplus (j-1)}$ be the other factors. Then we have an exact sequence
\begin{equation*}
\xymatrix{
0 \ar[r] & \sO_A(-D) \cong \sF \cap \sO_A \ar[r] & \sF \ar[r] & \frac{\sF}{\sF \cap \sO_A} \subseteq \sO^{\oplus (j-1)} \ar[r] & 0
}
\end{equation*}
This shows that $\det \sF = \sO_A(-D) \otimes \sO_A(-E)$ for another effective divisor $E$. Hence, since $\det \sF=0$, $D=0$ must hold. Therefore, $\sF \cap \sO_A = \sO_A$ for all factors, which implies that $\sF = \sO_A^{\oplus j}$. This concludes our claim.

From our claim we see that the inclusion $\sF \hookrightarrow \sT_A$ of foliations is given by the inclusion of a finite dimensional Lie algebra $\mathfrak{f} \hookrightarrow \Lie(A)$, where $\sF \cong \sO_A \otimes_k \mathfrak{f}$. Let $G \subseteq A$ be the kernel of the Frobenius morphism. Then, $\Lie(G) = \Lie(A)$. In particular,  $\mathfrak{f} \subseteq \Lie (G)$ induces a unique subgroup scheme $H \subseteq G$ according to \cite[Thm, p 139]{Mum70}, such that $A/H=A/\sF$. In particular, then $X$ is a abelian variety. 
\end{proof}


\begin{thebibliography}{ELMNP}
\bibitem[Abramovich94]{Abramovich94} {\sc D. Abramovich}, {\it Subvarieties of Semi-abelian varieties,} Compositio Math, Tome 90, 1 (1994) 37--52 
\bibitem[CZ13]{CZ13} {\sc Y. Chen and L. Zhang}, 	{\it The subadditivity of the Kodaira Dimension for Fibrations of Relative Dimension One in Positive Characteristics. } arXiv:1305.6024
\bibitem[dJ96]{dJ96}
{\sc A.~J. de~Jong}: \emph{Smoothness, semi-stability and alterations}, Inst.
  Hautes \'Etudes Sci. Publ. Math. (1996), no.~83, 51--93.
\bibitem[dJ97]{dJ97}
{\sc A.~J. de~Jong}: \emph{Families of curves and alterations}, Ann. Inst.
  Fourier (Grenoble) \textbf{47} (1997), no.~2, 599--621.
\bibitem[Eke88]{Eke88}
{\sc T.~Ekedahl}: \emph{Canonical models of surfaces of general type in
  positive characteristic}, Inst. Hautes \'Etudes Sci. Publ. Math. (1988),
  no.~67, 97--144.
\bibitem[GL87]{GL87} {\sc Green, M. and Lazarsfeld, R.}, {\it Deformation theory, generic vanishing theorems, and some conjectures of Enriques, Catanese and Beauville,} Invent. Math. 90 (1987), 389-407. MR0910207 (89b:32025)
\bibitem[GL91]{GL91} {\sc Green, M. and Lazarsfeld, R.}, {\it Higher obstructions to deforming cohomology groups of line bundles,} J. Amer. Math. Soc. 4, 1 (1991), 87-–103. MR1076513 (92i:32021)

\bibitem[HK15]{HK15} {\sc C. D. Hacon and S. Kov\'acs}, {\it Generic vanishing fails for singular varieties and in characteristic p>0.}
Cambridge, United Kingdom : Cambridge University Press, 2015. 
London Mathematical Society lecture note series, 417.
\bibitem[Hacon04]{Hacon04} {\sc C. D. Hacon}, {\it A derived category approach to generic vanishing.} J. Reine Angew. Math. 575 (2004), 173-–187.
\bibitem[HP13]{HP13} {\sc  C. D. Hacon, Zs. Patakfalvi}, {\it Generic vanishing in characteristic $p>0$ and the characterization of ordinary abelian varieties. }
 to apper in the American Journal of Mathematics (2013)
\bibitem[Kawamata81]{Kawamata81} {\sc Y. Kawamata}, {\it Characterization of Abelian Varieties,} Composition Math. Volume: 43, Issue: 2, page 253-276 (1981)
\bibitem[KV80]{KV80} {\sc Y. Kawamata and E. Viehweg}, {\it On a characterization of an abelian variety in the classification theory of algebraic varieties. }
Compositio Math. 41 (1980), no. 3, 355–-359.
\bibitem[Kee99]{Kee99}
{\sc S.~Keel}: \emph{Basepoint freeness for nef and big line bundles in
  positive characteristic}, Ann. of Math. (2) \textbf{149} (1999), no.~1,
  253--286. {\sf\scriptsize 1680559 (2000j:14011)}
\bibitem[Kol03]{Kol03}
{\sc J.~Koll{\'a}r}: \emph{Rationally connected varieties and fundamental
  groups}, Higher dimensional varieties and rational points ({B}udapest, 2001),
  Bolyai Soc. Math. Stud., vol.~12, Springer, Berlin, 2003, pp.~69--92.
\bibitem[KP15]{KP15}
{\sc S.~J. Kov{\'a}cs and {\relax Zs}.~Patakfalvi}: \emph{Projectivity of the
  moduli space of stable log-varieties and subadditvity of log-kodaira
  dimension}, arXiv:1503.02952 (2015).
%

\bibitem[Mum70]{Mum70}
{\sc D.~Mumford}: \emph{Abelian varieties}, Tata Institute of Fundamental
  Research Studies in Mathematics, No. 5, Published for the Tata Institute of
  Fundamental Research, Bombay, 1970.
\bibitem[Oort74]{Oort74}  {\sc F. Oort}, {\it Subvarieties of moduli spaces.} Invent. Math., 24:95–119, 1974.
\bibitem[RG71]{RG71}
{\sc M.~Raynaud and L.~Gruson}: \emph{Crit\`eres de platitude et de
  projectivit\'e. {T}echniques de ``platification'' d'un module}, Invent. Math.
  \textbf{13} (1971), 1--89.
\bibitem[Ueno73]{Ueno73}{\sc K. Ueno}, {\it Classification of algebraic varieties I}. Compositio Math. 27 1973, 277--324
\bibitem[SGA1]{SGA1} {\it Rev\'etements  \'etales et groupe fondamental (SGA 1)}. Documents Math \'ematiques (Paris)
[Mathematical  Documents  (Paris)],  3.  Soci \'et \'e  Math \'ematique  de  France,  Paris,  2003.
S \'eminaire de g \'eom \'etrie alg \'ebrique du Bois Marie 1960–61. [Algebraic Geometry Seminar  of  Bois  Marie  1960-61],  Directed  by  A.  Grothendieck,  With  two  papers  by  M.
Raynaud, Updated and annotated reprint of the 1971 original
[Lecture Notes in Math.,
224, Springer, Berlin; MR0354651 ]
\bibitem[Wang15]{Wang15} {\sc Y. Wang}, {\it Generic vanishing and classification of irregular surfaces in positive characteristics.} arXiv:1503.08384
\bibitem[WZ14]{WZ14} {\sc  A. Watson and Y. Zhang}, {\it On the generic vanishing theorem of Cartier modules.} arXiv:1404.2669
\bibitem[Wei11]{Wei11}
{\sc R.~Weissauer}: \emph{On subvarieties of abelian varieties with degenerate
  gauss mapping}, {\tt arXiv:1110.0095} (2011).
\end{thebibliography}
\end{document}